\documentclass[reqno]{amsart}
\usepackage{mathrsfs}
\usepackage{amsthm}
\usepackage{amscd}
\usepackage{amsmath}
\usepackage{amssymb}
\usepackage{extarrows}
\usepackage[colorlinks, linkcolor=blue]{hyperref}

\theoremstyle{definition}
\newtheorem{defn}{Definition}[section]

\theoremstyle{plain}
\newtheorem{thm}{Theorem}[section]
\newtheorem*{thm*}{Theorem}
\newtheorem*{thmA}{Theorem~A}
\newtheorem*{thmB}{Theorem~B}
\newtheorem{prop}[thm]{Proposition}
\newtheorem{lem}[thm]{Lemma}

\newtheorem{cor}[thm]{Corollary}

\theoremstyle{remark}
\newtheorem*{rmk}{Remark}

\newcommand{\mbbr}{\mathbb{R}}
\newcommand{\mbbt}{\mathbb{T}}
\newcommand{\mbbz}{\mathbb{Z}}

\newcommand{\mcla}{\mathcal{A}}
\newcommand{\mclb}{\mathcal{B}}
\newcommand{\mclm}{\mathcal{M}}

\newcommand{\mscm}{\mathscr{M}}

\newcommand{\tf}{\tilde{f}}

\newcommand{\dif}{\mathrm{d}\:\!}

\def \esssup {\mathrm{esssup}}
\def \essinf {\mathrm{essinf}}

\title{Regularity of calibrated sub-actions for circle expanding maps and Sturmian optimization}
\author{Rui Gao}
\address{Rui Gao: College of Mathematics, Sichuan University, Chengdu 610064, China}
\email{gaoruimath@scu.edu.cn}
\begin{document}

\begin{abstract}
  In this short and elementary note, we study some ergodic optimization problems for circle expanding maps. We first make an observation that if a function is not far from being convex, then its calibrated sub-actions are closer to convex functions in certain effective way. As an application of this simple observation,  for circle doubling map, we generalize a result of Bousch saying that translations of the cosine function are uniquely optimized by Sturmian measures. Our argument follows the mainline of Bousch's original proof, while some technical part is simplified by the observation mentioned above, and no numerical calculation is needed. 
\end{abstract}

\maketitle

  \section{Introduction}

Let $X$ be a compact metric space and let $C(X)$ denote the collection of real-valued continuous functions on $X$. Let $T:X\to X$ be a continuous map. For the topological dynamical system $(X,T)$, let $\mclm(X,T)$ denote the collection of $T$-invariant Borel probability measures on $X$. The {\bf ergodic optimization} problem of $(X,T)$ is concerned about, for given $f\in C(X)$, the extreme values (maximum or minimum) of $\mu\mapsto\int_X f\dif\mu$ for $\mu\in\mclm(X,T)$, and which measures attain the maximum or minimum. For an overview of the research topic of ergodic optimization, one may refer to Jenkinson's survey papers \cite{Jen07,Jen19}, and Bochi's survey paper \cite{Boc18}. See also Contreras, Lopes, and Thieullen \cite{CLT01}, and Garibaldi \cite{Gar17}.

Since
\[
\min_{\mu\in\mclm(X,T)}\int_X f\dif\mu=-\max_{\mu\in\mclm(X,T)}\int_X (-f)\dif\mu,
\]
to be definite, we shall mainly focus on the maximum problem and denote
\[
\beta(f):=\max_{\mu\in \mclm(X,T)}\int_X f\dif\mu,
\]
\[
\mclm_{max}(f):=\{\mu\in \mclm(X,T) : {\textstyle \int_X} f\dif\mu=\beta(f) \}.
\]
Measures contained in $\mclm_{max}(f)$ are called {\bf maximizing measures} of $f$ for $(X,T)$. Given $f\in C(X)$, a basic tool to study maximizing measures is usually referred to as  {\bf Ma\~{n}\'{e}'s lemma}, which aims at establishing existence of $g\in C(X)$, called a {\bf sub-action} of $f$, such that
\[
 \tf:= f+g-g\circ T\le \beta(f)
\]
holds. One of its advantages is based on a simple observation as follows. Once such $g$ exists, then $\beta(f)=\beta(\tf)$ and $\mclm_{max}(f)=\mclm_{max}(\tf)$; moreover, for any $\mu\in \mclm_{max}(f)$, $\tf=\beta(f)$ holds on the support of $\mu$. For a review of results on existence or construction of sub-actions, one may refer to \cite[\S~6]{Jen19}; see also \cite[Remark~12]{CLT01}.

Given $f\in C(X)$, a particular way to construct sub-action $g$ of $f$ is to solve the equation below:
\begin{equation}\label{eq:calib defn}
  g(x)+\beta(f)=\max_{Ty =x}(f(y)+g(y)), \quad \forall\, x\in X.
\end{equation}
A solution $g\in C(X)$ to Equation~\eqref{eq:calib defn} is automatically a sub-action of $f$, and it is called a {\bf calibrated sub-action} of $f$.  It is well known that when $(X,T)$ is expanding and $f$ is  H\"{o}lder continuous, for the one-parameter family of potentials $(tf)_{t>0}$, the zero-temperature limit points (as $t\to +\infty$) of their equilibrium states are contained in $\mclm_{max}(f)$; see, for example, \cite[Theorem~4.1]{Jen19} and references therein. Calibrated sub-actions arise naturally in this limit process; see, for example, \cite[Proposition~29]{CLT01}.  For expanding $(X,T)$ and H\"{o}lder continuous $f$, the first proofs on existence of calibrated sub-actions were given by, among others, Savchenko \cite{Sav99} and Bousch \cite{Bou00} with different approaches. Savchenko's proof is based on a zero-temperature limit argument mentioned above; while Bousch's proof is to treat \eqref{eq:calib defn} as a fixed point problem and to apply Schauder-Tychonoff fixed point theorem.

In the first half of this paper, for a smooth expanding system $(X,T)$, given an observable in $C(X)$ with certain smoothness, we are interested in the regularity (or smoothness) of its calibrated sub-actions. For simplicity, let us focus on the simplest case that $T:X\to X$ is a linear circle expanding map; more precisely, $X=\mbbt:=\mbbr/\mbbz$ is the unit circle, and $Tx=dx$ for $x\in \mbbt$, where $d\ge 2$ is an integer. Our main results in this part are Theorem~\ref{thm:calib} and  Corollary~\ref{cor:calib reg}, and they roughly read as follows.
\begin{thmA}
   For $f\in C(\mbbt)$ identified as a function on $\mbbr$ of period $1$, if $f$ can be written as a sum of a convex function and a quadratic function, then any calibrated sub-action $g$ of $f$ for the $x\mapsto dx$ system can be written in the same form. In particular, both one-sided derivatives of $g$ exist everywhere and only allow at most countably many discontinuities, and the set of points where $g$ is not differentiable is also at most countable.
\end{thmA}
It might be interesting to compare the statement above with the following result of Bousch and Jenkinson \cite[Theorem~B]{BouJen02}: for the $x\mapsto 2x$ system, there exists a real-analytic function such that any sub-action of it fails to be $C^1$.

In the second half of this paper, as an application of Theorem~\ref{thm:calib}, we shall generalize the a classic result of Bousch~\cite{Bou00} roughly stated below with a simpler proof.

\begin{thm*}[{Bousch~\cite[TH\'{E}OR\`{E}ME A]{Bou00}}]
 For each $\omega\in\mbbt$, the function $x\mapsto\cos 2\pi(x-\omega)$ has a unique maximizing measure for the $x\mapsto 2x$ system; moreover, the maximizing measure is Sturmian, i.e., supported on a semi-circle.
\end{thm*}
Our main technical result in the second half of this paper is Theorem~\ref{thm:sturm}.  As a corollary of it, we can easily prove Theorem~\ref{thm:cosine-like} as a generalization of \cite[TH\'{E}OR\`{E}ME A]{Bou00}. A slightly weaker and easier to read version of Theorem~\ref{thm:cosine-like} is as follows.

\begin{thmB}
  Let $f:\mbbt\to\mbbr$ be a $C^2$ function with the following properties, where $f$ is identified as a function on $\mbbr$ of period $1$:
  \begin{itemize}
    \item $f(x+\frac{1}{2})=-f(x)$,
    \item $f(x)=f(-x)$,
    \item $f''\le 0$ on $[-\frac{1}{4},\frac{1}{4}]$,
    \item $f(0)>\frac{1}{40}\cdot \max_{x\in\mbbt}f''(x)$.
  \end{itemize}
  Then for each $\omega\in\mbbt$,  the function $x\mapsto f(x-\omega)$ has a unique maximizing measure for the $x\mapsto 2x$ system; moreover, the maximizing measure is Sturmian.
\end{thmB}
To prove Theorem~\ref{thm:sturm}, we follow Bousch's proof of \cite[TH\'{E}OR\`{E}ME A]{Bou00}, while Theorem~\ref{thm:calib} is used to simplify some technical steps in Bousch's argument, especially \cite[LEMME TECHNIQUE 2]{Bou00}. No numerical calculation is needed in our argument.

The paper is organized as follows. In \S~\ref{se:calib}, we mainly prove Theorem~\ref{thm:calib}. The statements are given in \S~\ref{sse:calib result} and the proof is given in \S~\ref{sse:calib proof}. In \S~\ref{se:sturm}, we try to generalize \cite[TH\'{E}OR\`{E}ME A]{Bou00}; the main results are Theorem~\ref{thm:sturm} and its corollary Theorem~\ref{thm:cosine-like}. In \S~\ref{sse:Bousch}, we recall Bousch's ideas in his proof of \cite[TH\'{E}OR\`{E}ME A]{Bou00} that will be used for preparation.  In \S~\ref{sse:suff cond}, with the help of Theorem~\ref{thm:calib}, we prove Theorem~\ref{thm:sturm}. Finally,  Theorem~\ref{thm:cosine-like} is proved in \S~\ref{sse:cosine}.

The following notations or conventions are used throughout the paper.
\begin{itemize}
  \item For a topological space $X$, $C(X)$ denotes the collection of real-valued continuous functions on $X$
  \item For a metric space $X$, $Lip(X)$ denotes the collection of real-valued Lipschitz functions on $X$.
  \item For a (one-dimensional) smooth manifold (possibly with boundary) $X$, $C^k(X)$ denotes  the collection of real-valued functions that are continuously differentiable up to order $k$, where $k\ge 1$ is an integer.
  \item $\mbbt:=\mbbr/\mbbz$ denotes the unit circle. The following notations will be used when there is no ambiguity.
   \begin{itemize}
     \item For $x\in\mbbt$ and $y\in\mbbr$, $x + (y~\mathrm{mod}~1)\in\mbbt$ will be denoted by $x+y$ for short.
     \item For $a\in\mbbr$ and $b\in (a,a+1)$, the arc $\{x ~\mathrm{mod}~1: x\in [a,b]\}\subset \mbbt$ will be denoted by $[a,b]$ for short.
   \end{itemize}

  \item A function $f:\mbbt\to\mbbr$ will be frequently identified as a function on $\mbbr$ of period $1$.
  \item ``Lebesgue almost everywhere" is  abbreviated as ``a.e.".
\end{itemize}
\section{Regularity of calibrated sub-actions}\label{se:calib}

Throughout this section, the dynamical system is fixed to be
\[
T_d:\mbbt\to\mbbt, \quad x\mapsto dx,
\]
where $d\ge 2$ is an integer. We focus on the regularity of calibrated sub-actions for this system. Let us begin with introducing the following notation to express the equation \eqref{eq:calib defn} of calibrated sub-action in a more convenient way.  Define $\mscm_d:C(\mbbt)\to C(\mbbt)$ as follows:
\begin{equation}\label{eq:max op}
  \mscm_d f(x):=\max_{T_dy=x}f(y)=\max_{0\le k<d}f(\tfrac{x+k}{d}),\quad \forall\, f\in C(\mbbt),
\end{equation}
where in the expression $f(\tfrac{x+k}{d})$, $f$ is considered as a function on $\mbbr$ of period $1$. Note that by definition, for $f\in C(\mbbt)$,
\[
g  \ \text{is a calibrated sub-action of}\ f \iff g\in C(\mbbt) \text{ and } g+\beta(f)=\mscm_d (f+g).
\]
Also note that $Lip(\mbbt)$ is $\mscm_d$-invariant.  The following is well known.

\begin{lem}\label{lem:Lip cal}
  For every $f\in Lip(\mbbt)$  there exists $g\in  Lip(\mbbt)$ such that $g$ is a calibrated sub-action of $f$. Moreover, if $f$ has a unique maximizing measure, then $g$ is uniquely determined by $f$ up to an additive constant.
\end{lem}
For the existence of $g$, see, for example \cite[Proposition~29\,(iii)]{CLT01}. For the uniqueness part, see, for example \cite[LEMME~C]{Bou00}, which is stated for $d=2$, but the proof there is easily adapted to arbitrary $d\ge 2$ or more general systems; see also \cite[Proposition~6.7]{Gar17}.

\subsection{Statement of results}\label{sse:calib result}

To characterize regularity of calibrated sub-actions and to state Theorem~\ref{thm:calib} later on, we need the following quantity $\eta(\cdot)$ to measure how far a continuous function is  away from being convex.

\begin{defn}
  Given $f\in C(\mbbr)$, let $\eta(f)$ be the infimum  of $a>0$ (we adopt the convention that $\inf \varnothing=+\infty$) such that the following holds:
\[
f(x+\delta)+f(x-\delta)-2f(x) \ge -a\delta^2, \quad\forall\, x,\delta\in\mbbr.
\]
\end{defn}

By definition, the following is evident.
\begin{lem}\label{lem:eta convex}
 Given $f\in C(\mbbr)$ and $a\in[0,+\infty)$, $\eta(f)\le a$ iff $x\mapsto f(x)+\frac{a}{2} x^2$ is convex on $\mbbr$. 
\end{lem}

It will be useful to introduce the following notations closely related to $\eta(\cdot)$. Given $\delta>0$,  $x\in\mbbr$ and $f\in C(\mbbr)$, denote
\begin{equation}\label{eq:xi defn}
  \xi_\delta^x(f):=\max\{2f(x)-f(x+\delta)-f(x-\delta),0\},
\end{equation}
\begin{equation}\label{eq:xi* defn}
   \xi_\delta^*(f):=\sup_{x\in\mbbr}\xi_\delta^x(f).
\end{equation}
Note that by definition,
\begin{equation}\label{eq:eta xi}
  \eta(f)=\sup_{\delta>0}\delta^{-2}\xi_\delta^*(f), \quad \forall\, f\in C(\mbbr).
\end{equation}
Also note that $\eta(f)$, $\xi_\delta^x(f)$ and $\xi_\delta^*(f)$ are well defined for $f\in C(\mbbt)$ considered as a function on $\mbbr$ of period $1$.

Now we are ready to state our main result in this section.

\begin{thm}\label{thm:calib}
Given $d\ge 2$, let $f\in C(\mbbt)$ and suppose that there exists a calibrated sub-action $g\in C(\mbbt)$ of $f$ for the $x\mapsto dx$ system. Then we have:
   \begin{equation}\label{eq:xi cali}
     \xi_\delta^*(g)\le \xi_{d^{-1}\delta}^*(f)+\xi_{d^{-1}\delta}^*(g), \quad \forall\, \delta>0,
   \end{equation}
   \begin{equation}\label{eq:eta cali}
     \eta(g)\le (d^2-1)^{-1}\cdot\eta(f).
   \end{equation}
\end{thm}

The proof of Theorem~\ref{thm:calib} will be given in the next subsection. Let us state a corollary of it first, which indicates, under the regularity condition $\eta(f)<+\infty$,  how regular a calibrated sub-action of $f$ could be. It should be mentioned that, as the referee pointed out to the author, in \cite[Lemme~B]{Bou00} Bousch already showed that for any $f\in Lip(\mbbt)$ and any calibrated sub-action $g\in C(\mbbt)$ of $f$, $g$ is automatically Lipschitz and  $lip(g)\le \frac{lip(f)}{d-1}$ holds, where $lip(\cdot)$ denotes the Lipschitz constant of a function (although Bousch only focused on $d=2$, his argument works equally well for general $d\ge 2$). The corollary below asserts that, under the additional assumption $\eta(f)<+\infty$, we can say something more about the regularity of $g$.

\begin{cor}\label{cor:calib reg}
Given $d\ge 2$, let $f\in Lip(\mbbt)$ be with $\eta(f)<+\infty$ and let $g\in C(\mbbt)$ be a calibrated sub-action  of $f$ for the $x\mapsto dx$ system. Then the following hold.
  \begin{itemize}
    \item Both of the one-sided derivatives  $g_\pm'(x):=\lim\limits_{\Delta\searrow 0}\frac{g(x\pm \Delta)-g(x)}{\pm\Delta}$ of $g$ exist at each $x\in\mbbt$.
    \item $g_\pm'$ are of bounded variation on $\mbbt$; that is to say, identifying $g_\pm'$ as functions on $\mbbr$ of period $1$, $g_\pm'$ are of bounded variation on finite subintervals of $\mbbr$.
    \item Given $x\in\mbbt$, $g'(x)$ does not exist iff  $g_+'(x)>g_-'(x)$; the set $\{x\in\mbbt: g'(x) \text{ does not exist}\}$ is at most countable.
  \end{itemize}
\end{cor}

\begin{proof}
  By \eqref{eq:eta cali} in Theorem~\ref{thm:calib}, $\eta(g)<+\infty$. Then according to Lemma~\ref{lem:eta convex}, identifying $g$ as a function on $\mbbr$ of period $1$, $x\mapsto g(x)+\frac{\eta(g)}{2}x^2$ is convex on $\mbbr$. All the statements follow from basic properties of convex functions.
\end{proof}

\subsection{Proof of Theorem~\ref{thm:calib}}\label{sse:calib proof}

This subsection is devoted to the proof of Theorem~\ref{thm:calib}. Let us begin with some common elementary properties shared by $\eta(\cdot)$ and functionals defined in \eqref{eq:xi defn} and \eqref{eq:xi* defn}.

\begin{lem}\label{lem:cone}
Let $X$ be a topological space and recall that $C(X)$ denotes the vector space of real-valued continuous functions on $X$. Let $\Gamma$ denote the collection of $\gamma:C(X)\to [0,+\infty]$ satisfying the following assumptions (we adopt the convention $0\cdot(+\infty)=0$ below):
  \begin{itemize}
    \item $\gamma(f+g)\le \gamma(f)+\gamma(g)$;
    \item $\gamma(\max\{f,g\})\le \max\{\gamma(f),\gamma(g)\}$;
    \item $\gamma(af+b)=a\gamma(f)$ for $a\ge 0$, $b\in\mbbr$.
  \end{itemize}
Then $\Gamma$ has the following properties:
\begin{itemize}
  \item if $a,b\ge 0$ and $\alpha,\beta\in\Gamma$, then $a\alpha+b\beta\in\Gamma$;
  \item if $\Lambda\subset \Gamma$, then $\sup_{\gamma\in\Lambda}\gamma\in\Gamma$.
\end{itemize}
Moreover, for $X=\mbbr$, $\xi_\delta^x\in\Gamma$ for any $\delta>0$, $x\in\mbbr$. As a result, $\xi_\delta^*\in\Gamma$ for any $\delta>0$ and hence $\eta\in\Gamma$.
\end{lem}

\begin{proof}
  All the statements are more or less evident and can be easily proved by checking definition directly. Let us only verify the relatively less obvious fact that
  \[
  \xi_\delta^x(\max\{f,g\}) \le \max\{\xi_\delta^x(f),\xi_\delta^x(g)\},
  \]
  and the rest are left to the reader. To prove the inequality above, denote $h:=\max\{f,g\}$. If $h(x)=f(x)$, then
  \[
   \xi_\delta^x(h)=\max\{2f(x)- h(x-\delta)-h(x+\delta),0\} \le \xi_\delta^x(f),
  \]
  where the ``$\le$" is due to $f\le h$; otherwise, $h(x)=g(x)$ and the same argument shows that $\xi_\delta^x(h)\le \xi_\delta^x(g)$. The proof is done.
\end{proof}

The key ingredient to Theorem~\ref{thm:calib} is the simple fact below. Recall the operator $\mscm_d$ defined in \eqref{eq:max op}.
\begin{prop}\label{prop:max op}
Given $d\ge 2$ and $f\in C(\mbbt)$, the following hold:
   \begin{equation}\label{eq:max xi}
    \xi_\delta^*(\mscm_d f)\le \xi_{d^{-1}\delta}^*(f), \quad \forall\, \delta>0,
  \end{equation}
  \begin{equation}\label{eq:max eta}
    \eta( \mscm_d f)\le d^{-2}\cdot\eta(f).
  \end{equation}
\end{prop}

\begin{proof}
Identify $f$ as a function on $\mbbr$ of period $1$. For $0\le k<d$, let $f_k\in C(\mbbr)$ be defined by $f_k(x):=f(\frac{x+k}{d})$. By definition, $\mscm_d f= \max\limits_{0\le k<d}f_k$, and
  \[
  \xi_\delta^*(f_k) = \xi_{d^{-1}\delta}^*(f), \quad \forall\,0\le k<d,\,\delta>0.
  \]
It follows that
  \[
  \xi_\delta^*(\mscm_d f)=\xi_\delta^*(\max\limits_{0\le k<d}f_k)\le \max_{0\le k<d}\xi_\delta^*(f_k)=\xi_{d^{-1}\delta}^*(f),
  \]
where the ``$\le$" is due to Lemma~\ref{lem:cone}. This completes the proof of \eqref{eq:max xi}. \eqref{eq:max eta} follows from \eqref{eq:max xi} and \eqref{eq:eta xi} immediately.
\end{proof}

Now we are ready to prove Theorem~\ref{thm:calib}.
  \begin{proof}[Proof of Theorem~\ref{thm:calib}]
    Since $g+\beta(f)=\mscm_d(f+g)$, \eqref{eq:xi cali} follows from \eqref{eq:max xi} in Proposition~\ref{prop:max op} and Lemma~\ref{lem:cone} directly. To prove \eqref{eq:eta cali}, first note that by \eqref{eq:xi cali} and \eqref{eq:eta xi}, the following holds for any $\delta>0$:
    \[
   \xi_\delta^*(g)\le  \xi_{d^{-1}\delta}^*(f)+\xi_{d^{-1}\delta}^*(g)\le (d^{-1}\delta)^2\cdot\eta(f)+ \xi_{d^{-1}\delta}^*(g).
    \]
    Iterating this inequality repeatedly and noting that $\xi_{d^{-n}\delta}^*(g)\to 0$ as $n\to\infty$, we obtain that
    \[
    \xi_\delta^*(g) \le \sum_{n=1}^\infty (d^{-n}\delta)^2\cdot \eta(f) =\delta^2\cdot(d^2-1)^{-1}\cdot\eta(f).
    \]
   Since $\delta>0$ is arbitrary, \eqref{eq:eta cali} follows.
  \end{proof}

  \section{Sturmian optimization}\label{se:sturm}

In this section we fix $d=2$ and denote $Tx=2x$, $x\in\mbbt$.  For this system, we shall show that certain functions are uniquely maximized by Sturmian measures, which generalizes  the result of Bousch \cite[TH\'{E}OR\`{E}ME A]{Bou00}. The results are stated in  Theorem~\ref{thm:sturm} and its consequences Corollary~\ref{cor:sturm antisym} and Theorem~\ref{thm:cosine-like}. Recall that for any closed semicircle $S\subset \mbbt$ (i.e. $S=[\gamma,\gamma+\frac{1}{2}]$ for some $\gamma\in\mbbr$), there exists a unique $T$-invariant Borel probability measure supported on $S$, called the {\bf Sturmian measure} on $S$. Bullett and Sentenac \cite{BulSen94} studied Sturmian measures systematacially. See also \cite[\S~6]{Fog02} for connection of Sturmian measures with symbolic dynamics.

\subsection{Bousch's criterion of Sturmian condition}\label{sse:Bousch}

This subsection is devoted to reviewing the basic ideas of Bousch for his proof of \cite[TH\'{E}OR\`{E}ME A]{Bou00}.

\begin{defn}[{Bousch~\cite[Page~497]{Bou00}}]
  Let $f\in Lip(\mbbt)$ and let  $g\in Lip(\mbbt)$ be a calibrated sub-action of $f$ for the $x\mapsto 2x$ system. Denote
\begin{equation}\label{eq:funcR}
   R(x)=R_{f,g}(x):=f(x)+g(x)-f(x+\tfrac{1}{2})-g(x+\tfrac{1}{2}), \quad x\in\mbbt.
\end{equation}
If $\{x\in\mbbt :R(x)=0\}$ consists of a single pair of antipodal points, then we say that $(f,g)$ satisfies {\bf Sturmian condition}. We also say that $f$ satisfies Sturmian condition if $(f,g)$ satisfies Sturmian condition for some $g$.
\end{defn}

By definition, it is easy to see the following holds, which is a combination of the Proposition in page 497 of \cite{Bou00} and \cite[LEMME~C]{Bou00}.
\begin{lem}[{Bousch~\cite{Bou00}}]\label{lem:Sturmian cond}
  If $f\in Lip(\mbbt)$ satisfies Sturmian condition, then $f$ admits a unique maximizing measure. Moreover, this measure is a Sturmian measure, and up to an additive constant, there exists a unique calibrated  sub-action of $f$.
\end{lem}

\begin{rmk}
As a corollary, if $f\in Lip(\mbbt)$ satisfies Sturmian condition, then for the equilibrium state of $tf$ ($t>0$), its zero-temperature limit (as $t\to+\infty$) exists. See, for example, \cite[Theorem~4.1]{Jen19} for details.
\end{rmk}

The basic idea of Bousch \cite{Bou00} to verify Sturmian condition is summarized below.
\begin{lem}\label{lem:Bousch test}
Let $f\in Lip(\mbbt)$ and let $g\in Lip(\mbbt)$ be a calibrated sub-action of $f$. Assume that for $R$ defined by \eqref{eq:funcR}, there exist $a\in\mbbr$ and $b\in(a,a+\frac{1}{2})$ with the following properties:
  \begin{itemize}
    \item [(S1)]  $R>0$ on $[a,b]$,
    \item [(S2)]  $R$ is strictly decreasing on $[b,a+\frac{1}{2}]$.
  \end{itemize}
Then  $(f,g)$ satisfies Sturmian condition.
\end{lem}
Let us provide a proof for completeness.
\begin{proof}
  By definition, $R(x+\frac{1}{2})=-R(x)$, so that $\{x\in\mbbt: R(x)=0\}$ is non-empty and consists of pairs of antipodal points. On the other hand, the assumptions (S1) and (S2) imply that $\{x\in\mbbt: R(x)=0\}\cap [a,a+\frac{1}{2}]$ contains at most one point. The conclusion follows.
\end{proof}

To get an effective test for (S1) in Lemma~\ref{lem:Bousch test}, Bousch made the following observation in his proof of  \cite[LEMME TECHNIQUE 1]{Bou00}. Recall the notations introduced in \eqref{eq:xi defn} and  \eqref{eq:xi* defn}.
\begin{lem}
Let $f\in  Lip(\mbbt)$ and let $g\in Lip(\mbbt)$ be a calibrated sub-action of $f$. Then for any $n\ge 2$ and any $x\in\mbbt$, we have:
\begin{equation}\label{eq:R>0 reduced}
  -2\big(g(x)-g(x+\tfrac{1}{2})\big) \le \max_{y\in T^{-(n-1)}(x+\frac{1}{2})} \sum_{k=2}^{n} \xi_{2^{-k}}^{T^{n-k}y}(f) + \xi_{2^{-n}}^*(g) .
\end{equation}
\end{lem}

Let us reproduces Bousch's proof for completeness.
\begin{proof}
We follow the beginning part in the proof of \cite[LEMME TECHNIQUE 1]{Bou00}. Without loss of generality, assume that $\beta(f)=0$ for simplicity, so that \eqref{eq:calib defn} becomes
\[
g(x)=\max_{Ty=x}(f(y)+g(y)),\quad \forall\, x\in\mbbt.
\]
To prove \eqref{eq:R>0 reduced}, fix $x\in\mbbt$, denote $x_1=x+\tfrac{1}{2}$ and for each $n\ge 2$, choose $x_n\in T^{-1}x_{n-1}$ such that
\[
g(x_{n-1})=f(x_n)+g(x_n)
\]
inductively on $n$. Noting that $T(x_{n}\pm 2^{-n})=x_{n-1}\pm 2^{-n+1}$,  we have:
  \[
g(x_{n-1}+2^{-n+1}) \ge f(x_n+2^{-n})+ g(x_n+2^{-n}),
 \]
 \[
 g(x_{n-1}-2^{-n+1})\ge f(x_n-2^{-n})+ g(x_n-2^{-n}).
 \]
 It follows that for each $n\ge 2$,
 \[
 \xi_{2^{-(n-1)}}^{x_{n-1}}(g) \le \xi_{2^{-n}}^{x_{n}}(f) +\xi_{2^{-n}}^{x_{n}}(g),
 \]
and therefore
\[
-2(g(x)-g(x+\tfrac{1}{2})) = \xi_{2^{-1}}^{x_{1}}(g) \le \sum_{k=2}^n \xi_{2^{-k}}^{x_{k}}(f)+\xi_{2^{-n}}^{x_{n}}(g).
\]
The proof of \eqref{eq:R>0 reduced} is completed.
\end{proof}

\subsection{Sufficient conditions for Sturmian condition}\label{sse:suff cond}

In this subsection, based on Lemma~\ref{lem:Bousch test} and Theorem~\ref{thm:calib}, we shall prove Theorem~\ref{thm:sturm} and Corollary~\ref{cor:sturm antisym}, which provide sufficient conditions for Sturmian condition that are easy to check for specific observable $f$. Let us begin with a simple fact that will be used to verify (S2) in Lemma~\ref{lem:Bousch test}.

\begin{lem}\label{lem:deri des}
 The following holds for any $f\in Lip(\mbbr)$ with $\eta(f)<+\infty$ and any $\delta>0$:
  \[
  f'(x)-f'(x+\delta)\le \eta(f)\cdot\delta, \quad \text{a.e.}~x\in\mbbr.
  \]
\end{lem}
\begin{proof}
Let $\tilde{f}(x):=f(x)+\frac{\eta(f)}{2}x^2$. Then by Lemma~\ref{lem:eta convex}, $\tilde{f}$ is convex on $\mbbr$, so that for any $\delta>0$,
  \[
 f'(x)-f'(x+\delta)-\eta(f)\cdot\delta =\tilde{f}'(x)-\tilde{f}'(x+\delta)\le 0, \quad \text{a.e.}~x\in\mbbr.
  \]
\end{proof}

Our main technical result in this section is the following. Its statement might look complicated, so we shall present a simplified version in Corollary~\ref{cor:sturm antisym} and further deduce another result in Theorem~\ref{thm:cosine-like}.
\begin{thm}\label{thm:sturm}
  Let $f\in Lip(\mbbt)$ be with $\eta(f)<+\infty$. Suppose that there exist $a\in\mbbr$ and $b\in (a,a+\frac{1}{2})$ such that \eqref{eq:imply R>0} and \eqref{eq:imply R'<0} below hold:
  \begin{equation}\label{eq:imply R>0}
    f(x)-f(x+\tfrac{1}{2})-\tfrac{1}{2}\cdot\max_{y\in T^{-1}(x+\frac{1}{2})}\xi_{\frac{1}{4}}^y(f)>\tfrac{1}{96}\eta(f),\quad \forall\, x\in [a,b],
  \end{equation}
  \begin{equation}\label{eq:imply R'<0}
    f'(x)-f'(x+\tfrac{1}{2})<-\tfrac{1}{6}\eta(f), \quad\text{a.e.}~x\in [b,a+\tfrac{1}{2}].
  \end{equation}
Then $f$ satisfies Sturmian condition.
\end{thm}

\begin{proof}
Let $g\in Lip(\mbbt)$ be a calibrated sub-action of $f$. Note that by \eqref{eq:eta cali},  $\eta(g)\le \frac{1}{3}\eta(f)$. It suffices to verify the two assumptions (S1) and (S2) in Lemma~\ref{lem:Bousch test}.

{\bf Verifying (S1).}  Firstly, taking $n=2$ in \eqref{eq:R>0 reduced} yields that
\[
g(x)-g(x+\tfrac{1}{2})\ge  - \tfrac{1}{2}\cdot\max_{y\in T^{-1}(x+\frac{1}{2})}\xi_{\frac{1}{4}}^y(f) - \tfrac{1}{2}\cdot \xi_{\frac{1}{4}}^*(g).
\]
Secondly, due to \eqref{eq:eta xi} and $\eta(g)\le \frac{1}{3}\eta(f)$,
\[
\xi_{\frac{1}{4}}^*(g)\le \tfrac{1}{16}\eta(g)\le \tfrac{1}{48}\eta(f).
\]
Combining the inequalities in the two displayed lines above with  \eqref{eq:imply R>0} and the definition of $R$, assumption (S1) is verified.

{\bf Verifying (S2).} In suffices to show that $R'<0$ a.e. on $[b,a+\frac{1}{2}]$. Identify $g$ as a function on $\mbbr$ of period $1$. Applying Lemma~\ref{lem:deri des} to $g$ with $\delta=\frac{1}{2}$, we have:
\[
g'(x)-g'(x+\tfrac{1}{2})\le \tfrac{1}{2}\eta(g), \quad \text{a.e.}~x\in\mbbr.
\]
Combing this with $\eta(g)\le \frac{1}{3}\eta(f)$, \eqref{eq:imply R'<0} and the definition of $R$, the conclusion follows.
\end{proof}

To simplify the expressions in \eqref{eq:imply R>0} and \eqref{eq:imply R'<0}, let us add certain symmetry ((A0) in Definition~\ref{defn:antisym}) to $f$ and introduce the following families of functions. Then the statement of Theorem~\ref{thm:sturm} can be reduced to Corollary~\ref{cor:sturm antisym} below.

\begin{defn}\label{defn:antisym}
Given $a\in\mbbr$, $b\in (a,a+\frac{1}{2})$ and $v\in\mbbr$, let $\mcla_{a,b}^v$ denote the collection of $f\in Lip(\mbbt)$ such that the assumptions (A0), (A1), (A2) below hold.
  \begin{itemize}
    \item [(A0)] $f(x)+f(x+\frac{1}{2})=2v$ for $x\in\mbbt$.
    \item [(A1)] $2f(x)-v -\max_{y\in\mbbt} f(y)> \frac{1}{96}\eta(f)$ for $x\in [a,b]$.
    \item [(A2)] $f'(x)<- \frac{1}{12}\eta(f)$ for a.e. $x\in [b,a+\frac{1}{2}]$.
  \end{itemize}
Moreover, denote $\mcla_{a,b}:=\cup_{v\in\mbbr}\mcla_{a,b}^v$.
\end{defn}

\begin{rmk}
 By definition, if $f\in\mcla_{a,b}^v$, then for $f_\omega(x):=f(x-\omega)$, $f_\omega \in\mcla_{a+\omega,b+\omega}^v$ for any $\omega\in\mbbr$.  On the other hand, $\mcla_{a,b}$ is a ``cone" in the sense that if $f,g\in \mcla_{a,b}$ and $t>0$, then $tf, f+g\in \mcla_{a,b}$. This can be easily checked by definition and Lemma~\ref{lem:cone}.
\end{rmk}

In the statement below, by saying that a $T$-invariant measure is a {\bf minimizing measure} of $f\in C(\mbbt)$, we mean that it is a maximizing measure of $-f$.
\begin{cor}\label{cor:sturm antisym}
  For $\mcla_{a,b}$ defined in Definition~\ref{defn:antisym}, let $f\in \mcla_{a,b}$. Then for any $\omega\in \mbbr$, $f_\omega(x):=f(x-\omega)$ satisfies Sturmian condition. As a result, $f_\omega$ admits a unique maximizing measure and a unique minimizing measure, and both of them are Sturmian measures.
\end{cor}

\begin{rmk}
Results in \cite[Theorem~1]{ADJR12} and \cite[Theorem~4.1]{FSS21} are of the same style as Corollary~\ref{cor:sturm antisym} but cannot be recovered by it.
\end{rmk}

\begin{proof}
 Let $f\in \mcla_{a,b}^v$ for some $v\in\mbbr$. According to the remark ahead of the corollary, to show $f_\omega$ satisfies Sturmian condition, it suffices to prove that $f$ satisfies Sturmian condition. Firstly, since $f$ satisfies (A0), $\max_{y\in\mbbt} f(y)\ge v$ and 
  \[
  f(x)-f(x+\tfrac{1}{2})-\tfrac{1}{2}\cdot\xi_{\frac{1}{4}}^y(f)=2f(x)-v-\max\{f(y),v\},\quad \forall\,x,y\in\mbbt.
  \]
  Combining this with (A1),  \eqref{eq:imply R>0}  holds. Secondly, combining (A0) with (A2), \eqref{eq:imply R'<0} holds. Therefore, by Theorem~\ref{thm:sturm}, $f$ satisfies Sturmian condition.

  Finally, since $f_\omega$ satisfies Sturmian condition, by Lemma~\ref{lem:Sturmian cond}, it admits a unique maximizing measure which is Sturmian; since $f_{\omega}=2v-f_{\omega+\frac{1}{2}}$, the minimizing measure of $f_{\omega}$ is also unique and Sturmian.
\end{proof}


\subsection{A family of functions looking like cosine}\label{sse:cosine}
As a concrete application of Corollary~\ref{cor:sturm antisym}, we can recover the result of Bousch \cite[TH\'{E}OR\`{E}ME A]{Bou00}: for $f(x)=\cos 2\pi x$, $\eta(f)=4\pi^2$ and it is easy to check that $f\in\mcla_{-\frac{1}{8},\frac{1}{8}}$ (or  $f\in\mcla_{-\frac{1}{10},\frac{1}{10}}$).

To give a sufficient condition for Sturmian condition more explicit than Corollary~\ref{cor:sturm antisym}, by mimicking  behaviors of the cosine function, we add more properties to the observable $f$ and introduce the following family of functions. A sufficient condition will be given in Theorem~\ref{thm:cosine-like}.
\begin{defn}
  Let $\mclb$ denote the collection of $f\in Lip(\mbbt)$ with the following properties.
    \begin{itemize}
    \item $x\mapsto f(x)+f(x+\tfrac{1}{2})$ is constant on $\mbbt$.
    \item $f(-x)=f(x)$ for $x\in \mbbt$.
    \item $\eta(f)<+\infty$.
    \item $f$ is concave on $[-\frac{1}{4},\frac{1}{4}]$.
  \end{itemize}
\end{defn}

By definition, the following is evident, where $\esssup$ ($\essinf$) means the essential supremum (infimum) relative to Lebesgue measure.
\begin{lem}
$\mclb$ is a subset of $\{f\in C^1(\mbbt): f'\in Lip(\mbbt),f'(0)=0\}$. Moreover,
\[
\eta(f)=\esssup_{x\in\mbbt}f''(x) =-\essinf_{x\in\mbbt}f''(x), \quad \forall\, f\in\mclb.
\]
\end{lem}

\begin{thm}\label{thm:cosine-like}
Denote $\kappa:=\frac{7}{96}-\frac{\sqrt{3}}{36}>0$. Let $f\in \mclb$ and suppose that
\[
f(0)-f(\tfrac{1}{4})>\kappa\cdot\eta(f).
\]
Then $x\mapsto f(x-\omega)$ satisfies Sturmian condition for any $\omega\in \mbbr$.
\end{thm}

\begin{rmk}\mbox{}
\begin{itemize}
  \item $\frac{7}{96}-\frac{\sqrt{3}}{36}<0.02481<\frac{1}{40}$; on the other hand,
  \[
  \sup_{f\in \mclb} \frac{f(0)-f(\tfrac{1}{4})}{\eta(f)}=\frac{1}{32}= 0.03125,
  \]
  and the supremum is attained at $f\in\mclb$ defined by $f(x)=-x^2$, $x\in[0,\frac{1}{4}]$.
  \item For $f(x)=\cos 2\pi x\in\mclb$, $(f(0)-f(\frac{1}{4}))/\eta(f)=\frac{1}{4\pi^2}>0.02533$.
\end{itemize}

\end{rmk}

\begin{proof}
  Without loss of generality, assume that $f(\frac{1}{4})=0$ for simplicity. It suffices to find $c\in (0,\frac{1}{4})$ such that $f\in \mcla_{a,b}$ for $a=-c$, $b=c$. (A0) is automatically satisfied for $v=0$. Since  $f$ is decreasing on $[0,\frac{1}{2}]$ and since $f(-x)=f(x)$, (A1) is reduced to $2f(c)-f(0)>\frac{1}{96}\eta(f)$. Since $f'$ is decreasing on $[0,\frac{1}{4}]$ and since $f'(\frac{1}{2}-x)=f'(x)$, (A2) is reduced to $f'(c)<-\frac{1}{12}\eta(f)$. On the other hand, $h:=f(0)-f$ satisfies all the assumptions in Lemma~\ref{lem:sturm search} below, and (A1), (A2) are further reduced to (H1), (H2) in Lemma~\ref{lem:sturm search}. Then, applying Lemma~\ref{lem:sturm search} to this $h$, the conclusion follows.
\end{proof}

\begin{lem}\label{lem:sturm search}
  Let $h:[0,\frac{1}{4}]\to\mbbr$ satisfy the following.
\begin{itemize}
  \item $h$ is $C^1$, and $h'$ is Lipschitz and increasing.
  \item $h(0)=h'(0)=0$.
  \item $h(\frac{1}{4})>\kappa\cdot\eta$ for $\kappa:=\frac{7}{96}-\frac{\sqrt{3}}{36}$ and $\eta:=\esssup_{x\in [0,\frac{1}{4}]}h''(x)$.
\end{itemize}
Then there exists $c\in (0,\frac{1}{4})$ such that
\begin{itemize}
  \item [(H1)] $h(\frac{1}{4})-2h(c)>\frac{\eta}{96}$,
  \item [(H2)] $h'(c)>\frac{\eta}{12}$.
\end{itemize}
\end{lem}

\begin{proof}
  Without loss of generality, we may assume $\eta=1$, because all the assumptions of $h$ are positively homogeneous. Then $h''\le 1$ a.e.. Let $b\in [0,\frac{1}{4}]$ be maximal such that $h'(b)\le \frac{1}{12}$. Since $h'(0)=0$ and $h''\le 1$ a.e., $b>\frac{1}{12}$ and
\[
h'(x)\le
\left\{\begin{array}{ccc}
  x &,& 0\le x\le \tfrac{1}{12} \\
  \tfrac{1}{12}  &,& \tfrac{1}{12}< x\le b \\
  \tfrac{1}{12}+x-b &,& b< x\le \frac{1}{4}
\end{array}\right..
\]
It follows that
  \[
  h(b)= \int_0^b h'(x)\dif x \le \tfrac{b}{12}-\tfrac{1}{288},
  \]
\[
h(\tfrac{1}{4})=h(b)+\int_b^{\frac{1}{4}}h'(x)\dif x\le \tfrac{5}{288}+\tfrac{1}{2}\cdot(\tfrac{1}{4}-b)^2.
\]
Then from
\[
\tfrac{7}{96}-\tfrac{\sqrt{3}}{36}< h(\tfrac{1}{4})\le \tfrac{5}{288}+\tfrac{1}{2}\cdot(\tfrac{1}{4}-b)^2
\]
we deduce that $b<b_0:=\frac{5-2\sqrt{3}}{12}<\frac{1}{4}$, and hence $h(b)< \frac{b_0}{12}-\frac{1}{288}$. It follows that
\[
h(\tfrac{1}{4}) -2h(b)>\tfrac{7}{96}-\tfrac{\sqrt{3}}{36} - \tfrac{b_0}{6}+\tfrac{1}{144}=\tfrac{1}{96}.
\]
Then the conclusion follows by choosing $c>b$ sufficiently close to $b$.
\end{proof}

Motivated by Theorem~\ref{thm:cosine-like}, let us end this subsection with the following question that might be of interest.

{\bf Question.} What is the optimal lower bound of $\kappa>0$ such that if $f\in\mclb$ and  $f(0)-f(\frac{1}{4})>\kappa\cdot\eta(f)$ holds, then $x\mapsto f(x-\omega)$ satisfies Sturmian condition for any $\omega\in\mbbr$?

\section*{Acknowledgements}

The author thanks Zeng Lian for introducing the research topic of ergodic optimization to him. The author thanks Weixiao Shen for valuable suggestions. The author also thanks Bing Gao for helpful discussions. Finally the author would like to thank the anonymous referee for valuable remarks.

\bibliographystyle{plain}             
\bibliography{Sturmian_v2}

\end{document}